\documentclass{article}
\usepackage{amsmath,amssymb,amsthm}
\usepackage{parskip}
\usepackage{authblk}
\usepackage{tikz}
\usepackage[numbers,longnamesfirst]{natbib}
\usepackage{geometry}
\usepackage{multicol}
\usepackage{subcaption}
\usepackage[bookmarks=false,colorlinks=true,citecolor=blue]{hyperref}
\usepackage{orcidlink}

\newtheorem{theorem}{Theorem}{\bfseries}{\itshape}
\newtheorem{definition}{Definition}{\bfseries}{\itshape}
\newtheorem{lemma}{Lemma}{\bfseries}{\itshape}
\newtheorem{corollary}{Corollary}{\bfseries}{\itshape}
\newtheorem{observation}{Observation}{\bfseries}{\itshape}
\newtheorem{conjecture}{Conjecture}{\bfseries}{\itshape}
\newtheorem{claim}{Claim}{\bfseries}{\itshape}
{\bfseries}{\itshape}
\newtheorem{case}{Case}
\newtheorem{subcase}{Case}
\numberwithin{subcase}{case}
\newtheorem{subsubcase}{Case}
\numberwithin{subsubcase}{subcase}

\title{\textbf{Acyclic Edge Coloring of 3-sparse Graphs}}
\author[1]{\textbf{Nevil~Anto}\orcidlink{0000-0002-3379-1377}}
\author[1]{\textbf{Manu~Basavaraju}}
\author[1,2]{\textbf{Shashanka~Kulamarva}\orcidlink{0009-0002-2982-6044}}
\affil[1]{National Institute of Technology Karnataka, Surathkal-575025, India}
\affil[2]{School of Informatics, Kyoto University, Kyoto-6068501, Japan\protect\\ Email: \texttt{nevil.197cs005@nitk.edu.in, manub@nitk.edu.in, kulamarva.shashanka.3k@kyoto-u.ac.jp}}
\date{}

\begin{document}
	\maketitle
	\begin{abstract}
		\noindent A proper edge coloring of a graph without any bichromatic cycles is said to be an acyclic edge coloring of the graph. The \emph{acyclic chromatic index} of a graph $G$ denoted by $a'(G)$, is the minimum integer $k$ such that $G$ has an acyclic edge coloring with $k$ colors. Fiam\v{c}\'{\i}k conjectured that for a graph $G$ with maximum degree $\Delta$, $a'(G) \le \Delta+2$. A graph $G$ is said to be \emph{$3$-sparse} if every edge in $G$ is incident on at least one vertex of degree at most $3$. We prove the conjecture for the class of $3$-sparse graphs. Further, we give a stronger bound of $\Delta +1$, if there exists an edge $xy$ in the graph with $d_G(x)+ d_G(y) < \Delta +3$. When $\Delta >3$, the $3$-sparse graphs where no such edge exists is the set of bipartite graphs where one partition has vertices with degree exactly $3$  and the other partition has vertices with degree exactly $\Delta$.\\

		
	\end{abstract}
	
	\section{Introduction}
	All the graphs considered in this paper are finite, simple, and undirected. Let $G=(V,E)$ be a graph. A \emph{path} in $G$ is a sequence of distinct vertices in $V$ such that there is an edge between every pair of consecutive vertices in the sequence. When the first and last vertices in a path in $G$ are joined by an edge, then the resulting structure is called a \emph{cycle} in $G$. Proper edge coloring is the assignment of a color to each edge of the graph such that no adjacent edge receives the same color. The \emph{chromatic index} of $G$ (denoted by $\chi'(G)$) is the minimum number of colors required for a proper edge coloring of $G$.  We say that a cycle in $G$ is \emph{bichromatic} with respect to some proper  edge coloring $c$ if the edges of the cycle are colored with exactly $2$ colors. A proper edge coloring of $G$ without bichromatic cycles is said to be an \emph{acyclic} edge coloring of $G$. The minimum number of colors required for an acyclic edge coloring of $G$ is said to be the \emph{acyclic chromatic index} of a graph $G$  (denoted by $a'(G)$). The concept of acyclic coloring was coined by \citet{Grunbaum1973ACP} and it was extended to edge coloring by \citet{Fiamvcik1978AC}. The vertex analog of the acyclic chromatic index is helpful in determining bounds for some other parameters, such as star chromatic number \cite{Fertin2004StarCol} and oriented chromatic number \cite{Kostochka1997AC} of a graph. Both of these parameters have various practical uses, including wavelength routing in optical networks \cite{Amar2001Ntwk}.
	
	By Vizing's Theorem \cite{Diestel2017GT}, for a graph $G$ with maximum degree $\Delta$, we have $\Delta \le \chi'(G) \le \Delta+1$. Since an acyclic edge coloring is also a proper edge coloring, we have $a'(G) \ge \chi'(G) \ge \Delta$. In the seminal paper, \citet{Fiamvcik1978AC} (and independently \citet{Alon2001ACI}) conjectured that the upper bound for the acyclic chromatic index of any graph is $\Delta+2$.
	
	\begin{conjecture}[\cite{Fiamvcik1978AC}]\label{conj:ACI}
		For a graph $G$ with maximum degree $\Delta$, $a'(G) \le \Delta+2$.
	\end{conjecture}
	
	The conjecture remains unsolved for an arbitrary graph $G$, with the current best-known upper bound for $a'(G)$ being $3.569(\Delta-1)$, given by \citet{Fialho2020AECBound}. The fact that this bound is significantly distant from the conjectured bound highlights how challenging the problem is. However, the conjecture has been verified for certain special classes of graphs. \citet{Alon2001ACI} proved that there exists a constant $k$ such that $a'(G) \le \Delta+2$ for any graph $G$ with girth at least $k\Delta\log\Delta$.

    The acyclic chromatic index was exactly determined for some classes of graphs like series-parallel graphs when $\Delta \ne 4$ (\citet{Wang2011ACIK4MinorFree}), outerplanar graphs when $\Delta \ne 4$ (\citet{Hou2013AECOuterPlanarErr}, \citet{Hou2010AECOuterPlanar}), cubic graphs (\citet{Andersen2012AECCubic}), chordless graphs with $\Delta \ge 3$ (\citet{Basavaraju2023ACIChordless}), planar graphs with $\Delta \ge 4.2 \times 10 ^{14}$ (\citet{Cranston2019AECPlanar}) and planar graphs with girth at least 5 and $\Delta \ge 19$ (\citet{Basavaraju2011AECPlanar}). In the case of outerplanar graphs and series-parallel graphs, if $\Delta \ge 5$, then $a'(G)=\Delta$, and when $\Delta=3$, they \cite{Wang2011ACIK4MinorFree,Hou2010AECOuterPlanar} characterize the graphs that require $4$ colors for the acyclic edge coloring.
	
	A graph $G$ is said to be \emph{$k$-degenerate} if every subgraph of $G$ has a vertex of degree at most $k$. A graph $G$ is said to be \emph{$k$-sparse} if every edge in $G$ is incident on at least one vertex of degree at most $k$. One can see that the class of $k$-sparse graphs is a subclass of the class of $k$-degenerate graphs. \citet{Anto2024ACIDegenerate} proved that $a'(G) \leq \Delta + 5$ for $3$-degenerate graphs. In this paper, we study $3$-sparse graphs and prove Conjecture~\ref{conj:ACI} for $3$-sparse graphs. In particular, we have the following theorem.

	\begin{theorem}\label{thm:ACI3sprs}
		Let $G$ be a connected 3-sparse graph with maximum degree $\Delta$. If $G$ has an edge $xy$ with $d_G(x) + d_G(y) < \Delta + 3$, then $a'(G) \le \Delta+1$.
	\end{theorem}

	 Assuming Theorem~\ref{thm:ACI3sprs} holds, the set of $3$-sparse graphs yet to be verified to have $a'(G) \le \Delta+1$ is the  set of $3$-sparse graphs where every edge in the graph has one of its incident vertices with degree exactly $3$ and the other with degree exactly $\Delta$. When $\Delta >3$, the vertices with degree $3$ form an independent set in such graphs. Since the graph is $3$-sparse, the vertices with degree $\Delta$ also form an independent set when $\Delta >3$. Therefore, when $\Delta >3$, the set of $3$-sparse graphs excluded by Theorem~\ref{thm:ACI3sprs} is precisely the set of bipartite graphs where all the vertices with degree $3$ belong to one partition and all the vertices with degree $\Delta$ belong to the other partition. However, the following corollary shows that $\Delta+2$ colors are sufficient for acyclic edge coloring of any $3$-sparse graph, and therefore, Conjecture~\ref{conj:ACI} can be verified for all $3$-sparse graphs. Note that $K_{3,3}$, $K_4$ are $3$-sparse graphs with $\Delta =3 $ that have no edge $xy$ with $d_G(x) + d_G(y) < \Delta + 3$. Notice that $\Delta+2$ colors are required for acyclic edge coloring of $K_{3,3}$ and $K_4$. By the results of \citet{Andersen2012AECCubic} and \citet{Basavaraju2008AECSubcubic}, $a'(G) \leq 4$ when $\Delta \le 3$, unless $G$ is $K_4$ or $K_{3,3}$. Therefore, we can assume that $\Delta \ge 4$.

	\begin{corollary}\label{cor:ACI3sprs}
		If $G$ is a $3$-sparse graph, then $a'(G) \le \Delta+2$.
	\end{corollary}
	\begin{proof}
		We can see that if an edge is deleted from a $3$-sparse graph $G$, then the resulting graph is a $3$-sparse graph which contains an edge $xy$ with $d_G(x) + d_G(y) < \Delta +3 $ and hence, by Theorem~\ref{thm:ACI3sprs}, can be acyclically edge colored with $\Delta + 1$ colors. This trivially implies that $G$ can be acyclically edge colored with one additional color, as desired.

	\end{proof}
	
	\section{Preliminaries}
	Let $G=(V,E)$ be a graph with $n$ vertices and $m$ edges. The \emph{degree} of a vertex in a graph is the number of edges incident on that vertex. The degree of a vertex $v$ is represented as $d_G(v)$. We use notations $\delta(G)$ and $\Delta(G)$ to represent the minimum degree and the maximum degree of $G$ respectively. The \emph{edge degree} of an edge $e =uv$ is the number of edges incident on the endpoints of $e$ other than than $e$. Thus edge degree of $uv$ in $G$ =  $d_G(u) + d_G(v) -2 $. We say that a graph $G$ is \emph{$k$-edge-degenerate} if every subgraph of $G$  has an edge with edge degree at most $k$ or if every component in the subgraph is a trivial component. For any vertex $v \in V$, $N_G(v)$ is the set of all vertices in $V$ that are adjacent to the vertex $v$ in $G$. Throughout the paper, we ignore $G$ in these above notations whenever the graph $G$ is understood from the context.
	
	Let $X \subseteq E$ and $Y \subseteq V$, i.e., $X$ be a subset of $E$ and $Y$ be a subset of $V$. The subgraph of $G$ obtained by the vertex set $V$ and the edge set $E \setminus X$ is denoted as $G \setminus X$ and the subgraph of $G$ obtained by the vertex set $V \setminus Y$ and the edge set $E \setminus \{e \in E \mid \exists y \in Y \text{ such that } e \text{ is incident to } y\}$ is denoted as $G \setminus Y$. If either $X$ or $Y$ is a singleton set $\{u\}$, then we simply use $G \setminus u$ instead of $G \setminus \{u\}$. Further notations and definitions can be found in \cite{West2001IGT}. We use the word coloring instead of acyclic edge coloring at some obvious places when there is no ambiguity.

	First, we prove a lemma which is useful in our proof.
	
	\begin{lemma}\label{lem:edge_degenerate}
	Let $G$ be a connected $3$-sparse graph with maximum degree $\Delta$. Let $k = \Delta$. If $G$ has an edge $uv$ with edge degree at most $k$, then $G$ is $k$-edge degenerate.
	\end{lemma}
	
	\begin{proof}
	We have that the graph $G$ has an edge $uv$ with edge degree at most $k$. Suppose, for the sake of contradiction, assume that $G$ is not $k$-edge degenerate. Then, there exists some subgraph $H$ of $G$ with  $|E(H)| < | E(G)|$ and $H$ has no edge with an edge degree at most $k$. Therefore, every edge in $H$ has edge degree greater than $k$. Since $H$ is also $3$-sparse, note that edge degree of any edge is at most $\Delta + 1 = k+1$. Thus we can infer that every edge in $H$ has one of its endpoints with degree $3$ and the other endpoint with degree $\Delta(G) = k$.  
	
	Suppose $H$ is not connected and let $H_1$ be any non-trivial component in $H$. Then there exists a path $P\in E(G) \setminus E(H)$, connecting  $H_1$ and any component $H_2$ of $H$. Let $e \in E(P)$  be an edge which is incident on a vertex $x$ in $H_1$. Let $xz \in E(H_1$). Since edge degree of $xz >k$, $d_{H}(x) =  3$ or  $d_{H}(x) =  \Delta$. Since $\Delta$ is the maximum degree in $G$,  $d_{H}(x) \neq \Delta$. Otherwise, that would mean  $d_{G}(x) > \Delta  $. If $d_{H}(x) =  3$, then $d_{H}(z) =  \Delta$ and $d_{G}(x) \geq  4$. Presence of such an edge $xz$ in $G$ is a contradiction since $G$ is $3$-sparse. 

	Therefore, we may assume $H$ is connected. Let $e \in E(G) \setminus E(H)$ be an edge which is incident on a vertex $x$ in $H$. Since $G$ is connected, such an edge exists. Let $xz \in E(H$). Since edge degree of $xz >k$, $d_{H}(x) =  3$ or  $d_{H}(x) =  \Delta$. Since $\Delta$ is the maximum degree in $G$,  $d_{H}(x) \neq \Delta$. Otherwise, that would mean  $d_{G}(x) > \Delta  $. If $d_{H}(x) =  3$, then $d_{H}(z) =  \Delta$ and $d_{G}(x) \geq  4$. Presence of such an edge $xz$ in $G$ is a contradiction since $G$ is $3$-sparse. 
	 
	\end{proof}

	Further, the following definitions and lemmas are necessary for our proof. These were given by \citet{Basavaraju2010AEC2deg}.
	
	\begin{definition}[\cite{Basavaraju2010AEC2deg}]
		An edge coloring $f$ of a subgraph $H$ of a graph $G$ is called a \textbf{partial edge coloring} of $G$.
	\end{definition}

	Every graph is its own subgraph, making an edge coloring of a graph $G$ also a partial edge coloring of $G$. A partial coloring that obeys the rules of proper coloring is said to be proper. If bichromatic cycles are absent, a proper partial coloring is also acyclic. 

For an edge $e$ and partial coloring $f$, $f(e)$ may or may not be defined. So, whenever we use $f(e)$ for some edge $e$, we implicitly assume that $f(e)$ is defined. Given a partial edge coloring $f$ of $G$ and a vertex $u \in V$, we define $F_u(f) = \{f(uv) \mid v \in N_G(u)\}$. For any edge $xy \in E$, we define $F_{xy}(f) = F_y(f) \setminus \{f(xy)\}$. At certain places, we may choose not to specify the partial coloring $f$ if the context makes it sufficiently clear. For instance, with notations $F_u(f) $ and  $F_{xy}(f)$,  we may drop the name of the partial coloring $f$ from the brackets. In such situations, we will simply use  $F_u$ and $F_{xy}$ instead. It is easy to see that $F_{xy}$ and $F_{yx}$ are not the same. 
    	
	\begin{definition}[\cite{Basavaraju2010AEC2deg}]
A $(\mu,\nu)$-maximal bichromatic path with respect to a partial coloring $f$ of $G$ is a maximal path in $G$ consisting of edges that are colored using the colors $\mu$ and $\nu$ alternatingly. A $(\mu,\nu,x,y)$-\textbf{maximal bichromatic path} is a $(\mu,\nu)$-maximal bichromatic path that starts at the vertex $x$ with an edge colored with $\mu$ and ends at the vertex $y$.
	\end{definition}
	
	Now, we mention a lemma which follows from the definition of acyclic edge coloring. We assume this lemma implicitly further down the paper. This lemma was mentioned as a fact in \cite{Basavaraju2010AEC2deg}.
	
	\begin{lemma}[\cite{Basavaraju2010AEC2deg}]\label{lem:UniqueMaxBichPath}
Given a pair of colors $\mu$ and $\nu$ in a proper coloring $f$ of $G$, there is at most one $(\mu,\nu)$-maximal bichromatic path containing a particular vertex $y$ in $G$, with respect to $f$.
	\end{lemma}
	
	\begin{definition}[\cite{Basavaraju2010AEC2deg}]
		If the vertices $x$ and $y$ are adjacent in the graph $G$, then a  $(\mu,\nu,x,y)$-maximal bichromatic path in $G$, which ends at $y$ with an edge colored $\mu$, is said to be a $(\mu,\nu,xy)$-\textbf{critical path} in $G$.
	\end{definition}
	
	\begin{definition}[\cite{Basavaraju2010AEC2deg}]
		Let $f$ be a partial coloring of $G$. Let $w,x,y \in V$ and $wx,wy \in E$. A \textbf{color exchange} with respect to the edges $wx$ and $wy$ is defined as the process of obtaining a new partial coloring $g$ from the current partial coloring $f$ by exchanging the colors of the edges $wx$ and $wy$. The color exchange defines $g$ as follows. $g(wx)=f(wy)$, $g(wy)=f(wx)$ and the colors of all the other edges are same in both the partial colorings $f$ and $g$. The color exchange defined above is proper if the coloring obtained after the exchange is proper. The color exchange is said to be \textbf{valid} if the coloring obtained after the exchange is acyclic.
	\end{definition}

	\begin{definition}[\cite{Basavaraju2010AEC2deg}]
	    Let $f$ be a partial coloring of $G$. A \textbf{recoloring} $g$ of the partial coloring $f$ is defined as the process of obtaining a new partial coloring $g$ from the current partial coloring $f$ by assigning a different color to one or more edges. The recoloring is said to be proper (and valid) if the coloring $g$ obtained is proper (and acyclic).
	\end{definition}

	A color $\gamma$ is said to be a \emph{candidate} color for an edge $e$ in $G$ with respect to a partial coloring $f$ if none of the edges incident on $e$ are colored $\gamma$. A candidate color $\gamma$ is said to be \emph{valid} for an edge $e$ if assigning the color $\gamma$ to $e$ does not result in any new bichromatic cycle in $G$. \citet{Basavaraju2010AEC2deg} mentioned the following lemma as a fact since it is obvious.
	
	\begin{lemma}[\cite{Basavaraju2010AEC2deg}]\label{lem:ColorValidity}        Let $f$ be a partial coloring of $G$. A candidate color $\beta$ is not valid for an edge $e=(xy)$ if and only if there exists a color $\alpha \in F_{xy} \cap F_{yx}$ such that there is an $(\alpha,\beta,xy)$-critical path in $G$ with respect to the coloring $f$.
	\end{lemma}

	\section{Proof of Theorem~\ref{thm:ACI3sprs}}
	
	\renewcommand\qed{$\hfill\square$}

	\begin{proof}
	
	Let $G$ be a minimum counterexample to Theorem~\ref{thm:ACI3sprs}, with respect to the number of edges. It is given that $G$ has an edge $xy$ with $d_G(x) + d_G(y) <\Delta + 3$. Taking $t = \Delta$, we can infer from Lemma~\ref{lem:edge_degenerate} that $G$ is $t$-edge degenerate. We shall try to prove that $G$ can be acyclically edge colored using $(t+1)$ colors.

		Let the number of vertices and the number of edges of $G$ be represented by $n$ and $m$ respectively. As $G$ is $3$-sparse, either $d_G(x) \le 3$ or $d_G(y) \le 3$. Without loss of generality, assume that $d_G(x) \le 3$. Let $G' =  G \setminus xy$, i.e., a graph formed by deleting the edge $xy$ from $G$. Suppose $G'$ is not connected. Since $G$ is a minimum counter example and $E(G') < E(G)$,  each of the two connected components can be colored with at most $t+1$ colors.  Recall that $F_{yx}$ and $F_{xy}$ denote the set of colors on the edges incident on $x$ and $y$ in $G'$ respectively. As the edge degree of $xy$ is at most $t$, there is at least one candidate color that can be assigned to the edge $xy$ to get a proper edge coloring of $G$. Since removal of the edge $xy$ disconnected the graph, the assignment of color to edge $xy$ can not create any bichromatic cycles. Hence coloring is acylic also. Therefore we may assume that $G'$ is connected. By Lemma~\ref{lem:edge_degenerate}, $G'$ is also a $3$-sparse graph with an edge having edge degree at most $t$. Note that $G'$ has less than $m$ edges and since $G$ is a minimum counterexample, $G'$ has an acyclic edge coloring $g$ with at most $t+1$ colors. Let $S$ denote the set of candidate colors for the edge $xy$, i.e., colors not in $F_{xy} \cup F_{yx}$. 
		
\begin{observation}
If a color $\mu_k \in S$ is valid for the edge $xy$, then we can get a valid coloring for the minimum counterexample $G$. Therefore, no color in $S$ is valid for the edge $xy$ in $G$.
\end{observation}

We first prove that in the minimum counter example $G$, both $x$ and $y$ can not have degree  at most $3$ at the same time.

		\begin{lemma}\label{lem:ge3}
			$G$ being a minimum counterexample, can not have $d_G(x) \leq 3$ and $d_G(y) \leq 3$.
		\end{lemma}

		\begin{proof}
    		Suppose  $d_G(x) \leq 3$ and $d_G(y) \leq 3$. If $d_G(x) =2$, let $x_1$ be the other neighbor of $x$, besides $y$. Similarly let $y_1$ be the other neighbor of $y$ if $d_G(y) =2$. If $d_G(x) =3$, let $x_1$ and $x_2$ be the other neighbors of $x$  besides $y$. Similarly, let $y_1$ and $y_2$ be the other neighbors of $y$, besides $x$,  when $d_G(y) =3 $. Depending on the value of $|F_{xy} \cap F_{yx}|$, we have the following cases:
    		
 			\begin{case}
     			$|F_{xy} \cap F_{yx}| = 0$
 			\end{case}\label{case:atmost3_intersection0}

There exists at least one candidate color $\mu_i \notin F_{xy} \cup F_{yx}$ that can be assigned to the edge $xy$ to get a proper coloring of $G$. Since $|F_{xy} \cap F_{yx}| = 0$, this assignment of a color to $xy$ doesn't produce any  bichromatic cycles. Hence the coloring is valid.
			\begin{case}\label{case:atmost3_intersection2}
			    $|F_{xy} \cap F_{yx}| = 2$
			\end{case}

			Let $g(xx_1) = g(yy_1) = 1$ and $g(xx_2) = g(yy_2) =2$. Let the set of candidate colors for $xy$ be $S = \{ \mu_1, \mu_2, ..., \mu_{t-1} \}$. Based on the presence of colors $1$ and $2$ in $F_{yy_1}$ and $F_{yy_2}$, we have the following subcases:
			
			\begin{subcase}
			    $1 \in F_{yy_2} $ and $2 \in F_{yy_1}$. 
			\end{subcase}
			
	Since $|S| = t-1$, we can infer that there exists a color $\mu_i \notin F_{yy_1}$ and a color $ \mu_j \notin F_{yy_2}$, where $\mu_i, \mu_j \in S$. A color $\mu_k \in S$ is invalid for the edge $xy$ only if there exists a $( 1,\mu_k, xy)$-critical path or $( 2, \mu_k, xy)$-critical path or both. Since $\mu_i \notin F_{yy_1}$, there exists a $ ( 2, \mu_i, xy)$-critical path.  We define a recoloring $f_0$ of $g$: The edge $yy_1$ is assigned the color $\mu_i$. By Lemma~\ref{lem:UniqueMaxBichPath}, there can exist at most one  $(2, \mu_i)$-maximal bichromatic path through vertex $y$. As we have seen, the $(2, \mu_i)$-maximal bichromatic path that passes through $y$ ends at $x$. Therefore, the recoloring is valid. Assigning the edge $xy$ the color $\mu_j$ gives a valid coloring for $G$ because $ \mu_j \notin F_{yy_2}$. This is a contradiction to our assumption that $G$ is a minimum counterexample.

\begin{subcase}
     ($1 \in F_{yy_2} $ and $2 \notin F_{yy_1}$) or  ($1 \notin F_{yy_2} $ and $2 \in F_{yy_1}$) 
 \end{subcase}
 Since both cases are symmetrical, we will prove the case where $1 \in F_{yy_2} $ and $2 \notin F_{yy_1}$.
 Since $1 \in F_{yy_2} $, there exists a color $\mu_i \notin F_{yy_2}$. As $\mu_i$ is not a valid color for the edge $xy$, a $(1, \mu_i, xy)$-critical path exists. Also, $\mu_i \in F_{xx_1}$. We define a recoloring $f_0$ of $g$: edge $yy_2$ is recolored with color $\mu_i$. Notice that a $(1,\mu_i$)-bichromatic cycle is not created in $f_0$. This is because by   Lemma~\ref{lem:UniqueMaxBichPath}, there exists no $(1, \mu_i, yy_2)$-critical path in $f_1$ since there exists a $(1, \mu_i)$-maximal bichromatic path passing through $y$ and ending at $x$. Therefore, the coloring is valid. We can attempt to extend the color $f_0$ to get a valid coloring for $G$ by assigning one of the colors in $S \setminus \mu_i$ to the edge $xy$. If none of these colors are valid, then there exists a $( 1,\mu_k, xy)$-critical path, for every $\mu_k \in S \setminus \mu_i$ with respect to $f_0$. Since $\mu_i \in F_{xx_1}$  also, we can conclude that $F_{xx_1} = S $. Note that, $2 \notin F_{xx_1}$.

 Suppose $1 \notin F_{xx_2}$. We define the following recoloring $f_1$ of $g$: The colors on the edges $xx_2$ and $xx_1$ are exchanged i.e, the edge $xx_1$ is assigned the color $2$ and the edge $xx_2$ is assigned the color $1$. Clearly, the coloring is valid because $1 \notin F_{xx_1}$ and $2 \notin F_{xx_2}$. Assigning color $\mu_i$  to the edge $xy$ gives a valid coloring for $G$. This is because the $(1, \mu_i, x, y)$-maximal bichromatic path in $g$ goes through $x_1$ and  $\mu_i \notin F_{yy_2}$ . By Lemma~\ref{lem:UniqueMaxBichPath}, at most one $(1, \mu_i)$-maximal bichromatic path can go through $y$.
 
  Therefore, we can assume that $1 \in F_{xx_2}$. Then there exists a color $\mu_j \notin F_{xx_2}$. We can define the following recoloring $f_2$ of $g$: recolor the edge $xx_2$ with the color $\mu_j$, the edge $yy_2$ with the color $\mu_i$. Recall that assigning color $\mu_i$ to the edge $yy_2$ is valid. By Lemma~\ref{lem:UniqueMaxBichPath}, at most one $(1,\mu_j)$-maximal bichromatic path can pass through $x$ and the $(1,\mu_j)$-maximal bichromatic path through $x$ goes to $y$. Therefore, coloring the edge $xx_2$ with the color $\mu_j$ is also valid. The color $2$ is now a candidate color for assignment to edge $xy$. Since $2 \notin F_{xx_1}$ and $2 \notin F_{xx_2}$, this gives us a valid coloring of $G$, a contradiction to our assumption that $G$ was a minimum counterexample.

 \begin{subcase}
    $1 \notin F_{yy_2} $ and $2 \notin F_{yy_1}$ 
 \end{subcase}
 Given coloring $g$ for $G'$, we can attempt to extend the coloring $g$ to get a coloring for $G$ by assigning one of the candidate colors in $S$ to the edge $xy$. Since no color in $S$ is valid for the edge $xy$, for every color $\mu_i \in S$, there exists either a $(1,\mu_i, xy)$-critical path or a $(2,\mu_i, xy)$-critical path or both. To be specific, the $(1,\mu_i, x, y)$-maximal bichromatic paths pass through $x_1$ and $y_1$. The $(2,\mu_i, x, y)$-maximal bichromatic paths  pass through $x_2$ and $y_2$.  Suppose there exists $\gamma$ number of $(1,\mu_i,x, y)$-maximal bichromatic paths that pass through $x_1$ (Note that all $\mu_i$ are distinct. See Lemma~\ref{lem:UniqueMaxBichPath}). Since all the $t-1$ colors in $S$ are not valid for the edge $xy$, there exists at least $((t-1) - \gamma)$ number of $(2,\mu_i,x, y)$-maximal bichromatic paths that pass through $x_2$.
 
 Now we can attempt a recoloring $f_0$ of $g$ by exchanging colors on the edges $yy_1$ and $yy_2$. Since  $1 \notin F_{yy_2} $ and $2 \notin F_{yy_1}$, the recoloring is proper.  Since  $1 \notin F_{yy_1} $ and $2 \notin F_{yy_2}$, the recoloring is valid.

\begin{figure}[h!]
			\centering
			\begin{tikzpicture}[scale = 0.8]
				\begin{scope}[every node/.style={circle,draw,inner sep=0pt, minimum size=4.5ex}]
					\node (x) at (-1,1) {$x$};
                    \node (y) at (2,1) {$y$};
					\node (y1) at (4,2) {$y_1$};
                    \node (y2) at (4, 0) {$y_2$};
					\node (x1) at (-3,2) {$x_1$};
					\node (x2) at (-3,0) {$x_2$};
     	              \node (z1) at (6,3.5) {$z_1$};
				    \node (z2) at (8,2) {$z_{t-1}$};
                    \node (w1) at (6, -2) { $u_a$};
                    \node (w2) at (8, 0){};
				\end{scope}
			
		              \node (zi) at (6, 3) { };
                    \node (zj) at (6.5, 2.6){};
                    \node (z1a) at (7.5,6){};
                    \node (z1b) at (8,5){};
                    \node (z2b) at (10.5,2.7){};
                    \node (z2a) at (10,3.5){};
                    
                    \node (wi) at (6, -1.25) { };
                    \node (wj) at (6.5, -0.75){};

                     \node (w2a) at (7.5,-4){};
                    \node (w2b) at (8,-3.5){};
                    \node (w1b) at (10.5,0.7){};
                    \node (w1a) at (10,1.5){};
                    
				\draw [dotted] (x) to (y);
                \draw [dash dot]  (zi) to (y1);
                \draw [dash dot]  (zj) to (y1);
                 \draw (y) to  node[above]{$1$} (y1);
                \draw (y) to  node[above]{$2$} (y2);
				\draw (x) to node[above]{$1$} (x1);
				\draw (x) to node[below]{$2$} (x2);
				\draw (y1) to node[above]{$\mu_1$} (z1);
				\draw (y1) to node[below]{$\mu_{t-1}$} (z2);
				
				\draw (z1) to node[above]{$1$} (z1a);
				\draw (z1) to node[above]{$2$} (z1b);
                \draw (z2) to node[above]{$1$} (z2a);
				\draw (z2) to node[above]{$2$} (z2b);

                \draw (y2) to node[below]{$\mu_1$} (w1);
				\draw (y2) to node[above]{$\mu_{t-1}$} (w2);
    
				\draw [dash dot]  (wi) to (y2);
                \draw [dash dot]  (wj) to (y2);

                \draw   (w1b) to node[above]{$2$} (w2);
                \draw   (w1a) to node[above]{$1$} (w2);

                 \draw  (w2b)to node[above]{$1$}   (w1);
                \draw  (w2a) to node[above]{$2$} (w1);
			\end{tikzpicture}
			\caption{Case 2.3. The neighbourhood of edge $xy$}
			\label{fig:case2.2}
		\end{figure}
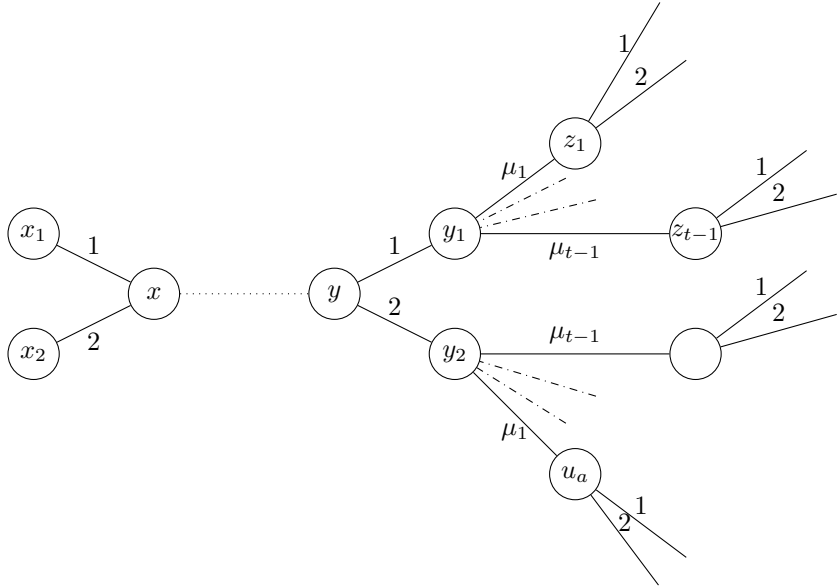
  
\begin{claim} The number of $(1,\mu_i)$-bichromatic paths passing through $x_1$ is $t-1$ and the number of $(2,\mu_i)$-bichromatic paths passing through $x_2$ is $t-1$.

\end{claim} \label{claim:tminus1}
\begin{proof}
We can attempt to extend the coloring $f_0$ to get a coloring for $G$ by assigning colors from the set $S$ to $xy$. If any of the $t-1$ colors in $S$ is valid for the edge $xy$, then we can successfully extend the coloring to get a valid coloring for $G$. Therefore, we may assume that there exists either a $(1,\mu_i, x,y)$-maximal bichromatic path or a $(2,\mu_i, x,y)$-maximal bichromatic path or both for every $\mu_i \in S$. The $(1,\mu_i,x,y)$-maximal bichromatic paths pass  through $y_2$ and $x_1$ with respect to coloring $f_0$. The $(2,\mu_i,x,y)$-maximal bichromatic paths pass through $y_1$ and $x_2$ with respect to coloring $f_0$. Recall that from our failed attempt to extend coloring $g$, we saw that there exist at least $((t-1) - \gamma)$ number of $(2,\mu_i,x, y_2)$-maximal bichromatic paths with respect to coloring $f_0$ that pass through vertex $x_2$. Therefore, by Lemma~\ref{lem:UniqueMaxBichPath}, no color $\mu_k$ from this set of  $((t-1) - \gamma)$ colors can have a second $(2,\mu_k)$-maximal bichromatic path that passes through $x_2$. We can see that at most $\gamma$ number of other  $(2,\mu_i)$-bichromatic paths can go from $x$ to $y_1$ through $x_2$. Since $t-1$ colors from $S$ are invalid for the edge $xy$ with respect to coloring $f_0$, at least $((t-1)-\gamma)$ number of $(1,\mu_i)$-bichromatic paths go from $x$ to $y_2$ through $x_1$. By Lemma~\ref{lem:UniqueMaxBichPath}, these $((t-1)-\gamma)$ colors are different from the $\gamma$ colors that we know to have a $(1,\mu_i$)-maximal bichromatic path passing through $x_1$. Therefore, exactly $\gamma$ number of other  $(2,\mu_i)$-bichromatic paths  go from $x$ to $y_1$ through $x_2$ and exactly $((t-1)-\gamma)$ number of $(1,\mu_i)$-bichromatic paths go from $x$ to $y_2$ through $x_1$. So, from our two failed attempts to extend the colorings $f_0$ and $g$ to get a valid coloring for $G$, we can infer that there are $t-1$ number of  $(1,\mu_i)$-bichromatic paths passing through $x_1$ and  $t-1$ number of  $(2,\mu_i)$-bichromatic paths passing through $x_2$, as claimed. 
\end{proof}

Therefore, $F_{xx_2} =  F_{xx_1} = S $ and $1 \notin F_{xx_2} $ and $2 \notin F_{xx_1}$.
 
\begin{claim} $F_{y_1z_i} = \{1,2\}$, for every neighbor $z_i \in N_G(y_1) \setminus y$ and $F_{y_2u_i} = \{1,2\}$, for every neighbor $u_i \in N_G(y_2) \setminus y$. 
\end{claim}
\begin{proof}
Since  $1 \notin F_{xx_2} $ and $2 \notin F_{xx_1}$,  exchanging the colors of the edges $xx_1$ and $xx_2$ to get a new valid coloring $f_1$ of $g$ is possible. We can again attempt to extend the coloring $f_1$  to get a valid coloring of the graph $G$ by assigning colors from  $S$ to $xy$. 

Recall that from our failed attempt to extend coloring $g$, we saw that there exist at least $((t-1) - \gamma)$ number of $(2,\mu_i,y, x_2)$-maximal bichromatic paths with respect to coloring $f_1$ that pass through vertex $y_2$. Therefore, by Lemma~\ref{lem:UniqueMaxBichPath}, no color $\mu_k$ from this set of  $((t-1) - \gamma)$ colors can have a second $(2,\mu_k)$-maximal bichromatic path that passes through $y_2$. We can see that at most $\gamma$ number of $(2,\mu_i)$-bichromatic paths can go from $y$ to $x_1$ through $y_2$. Since $t-1$ colors from $S$ are invalid for the edge $xy$ with respect to coloring $f_1$, at least $((t-1)-\gamma)$ number of $(1,\mu_i)$-maximal bichromatic paths go from  $y$ to $x_2$ through $y_1$. By Lemma~\ref{lem:UniqueMaxBichPath}, these $((t-1)-\gamma)$ colors are different from the $\gamma$ colors that we know to have a $(1,\mu_i$)-maximal bichromatic path passing through $y_1$. Therefore, exactly  $((t-1)-\gamma)$ number of $(1,\mu_i)$-bichromatic paths go from  $y$ to $x_2$ through $y_1$ and exactly $\gamma$ number of $(2,\mu_i)$-bichromatic paths can go from $y$ to $x_1$ through $y_2$ with respect to coloring $f_1$. So, from our two failed attempts to extend the colorings $g$ and $f_1$ to get a valid coloring for $G$, we can infer that there are $t-1$  number of $(1, \mu_i)$-bichromatic paths passing through $y_1$ and  $t-1$ number of $(2,\mu_i)$-bichromatic paths passing through $y_2$.

Next, we attempt the following recoloring $f_2$ of $g$: The edges $xx_1$ and $yy_1$ are assigned color $2$, while the edges $yy_2$ and $xx_2$ are assigned color $1$. We can see that the recoloring is valid. As before, we could then attempt to extend the coloring $f_2$ to get a valid coloring for $G$ by assigning a candidate color from $S$ to the edge $xy$. A color  $\mu_i \in S$ is invalid only if there exists a $(1, \mu_i,x,y)$-maximal bichromatic path that passes through $x_2$ and $y_2$ or a  $(2, \mu_i,x,y)$-maximal bichromatic path that passes through $x_1$ and $y_1$. Recall that, from our failed attempt to extend coloring $f_0$, we saw that there exist  $\gamma$ number of $(2,\mu_i, y, x_2)$-maximal bichromatic paths  that pass through $y_1$  with respect to coloring $f_2$.  Note that by Lemma~\ref{lem:UniqueMaxBichPath}, for any color $\mu_k$  from this set of $\gamma$ colors there can not exist a different $(2,\mu_k)$-maximal bichromatic path passing through $y_1$. Also, from our failed attempt to extend coloring $f_0$, we saw that there exist $((t-1) - \gamma)$ number of $(1,\mu_i,y, x_1)$-maximal bichromatic paths that pass through $y_2$ with respect to coloring $f_2$. Again, by Lemma~\ref{lem:UniqueMaxBichPath}, for any color $\mu_k$  from this set of $((t-1) - \gamma)$  colors, there can not exist a different $(1,\mu_k)$-maximal bichromatic path passing through $y_2$.

Since $|F_{yy_1}| \leq t-1$ and $ |F_{yy_2}| \leq t-1$, and because  all the $t-1$ colors from $S$ are invalid for the edge $xy$, $((t-1)-\gamma)$ number of $(2, \mu_i,x,y)$-maximal bichromatic paths exist that pass through $x_1$, $y_1$ and  $\gamma$ number of $(1, \mu_i,x,y)$-maximal bichromatic paths exist that pass through $x_2$ , $y_2$. Therefore, from the two failed attempts to extend colorings $f_2$ and $f_0$, we can infer that there are $t-1$  number of $(1,\mu_i)$-bichromatic paths passing through $y_2$ and  $t-1$ number of $(2, \mu_i)$-bichromatic paths passing through $y_1$.

From this we can conclude that $F_{y_1z_i} = \{1,2\}$, for every neighbor $z_i \in N_G(y_1) \setminus y$ and $F_{y_2u_i} = \{1,2\}$, for every neighbor $u_i \in N_G(y_2) \setminus y$ ( See Figure~\ref{fig:case2.2}).
\end{proof}

 Suppose, $g(y_1z_1) = \mu_1$ and $g(y_1z_2) = \mu_2$.   We can then define the following recoloring $f_3$ of $g$:

The colors of the edges $y_1z_1$ and $y_1z_2$ are exchanged, i.e $y_1z_2$ is assigned color $\mu_1$ and $y_1z_1$ is assigned color $\mu_2$ and colors of edges $y_2u_a$ and $y_2u_b$  are exchanged, where $y_2u_a$ and $y_2u_b$ are two edges with $g(y_2u_a) = \mu_1$ and $g(y_2u_b) = \mu_2$. This breaks the $(1, \mu_1, x, y)$, $(1, \mu_2, x, y)$, $(2, \mu_1, x, y)$ and $(2, \mu_2, x, y)$  critical paths. This recoloring is valid because $\{1,2\} \notin F_{yy_1}$ and $\{1,2\} \notin F_{yy_2}$.  Assigning either $\mu_1$ or $\mu_2$ to edge $xy$ gives a valid coloring for $G$.

 
 \begin{case}
     $|F_{xy} \cap F_{yx}| = 1$
 \end{case}\label{case:atmost3_intersection1}

Suppose $d_G(x) = d_G(y) =2$.  $|F_{xy} \cup F_{yx} \cup F_{yy_1}| \leq t$. There exists at least one candidate color that can be assigned to the edge $xy$ to get a valid coloring for $G$. Therefore, at most one of the vertices $x$ or $y$ can have degree $2$. Without loss of generality, let $d_G(x) =2$ and $d_G(y) =3$. Let $g(xx_1) = g(yy_1) = 1$ and $g(yy_2) = 2$. Since  $|F_{xy} \cap F_{yx}| =1$, $|F_{xy} \cup F_{yx}| =2$. The set of candidate colors $S$ has $t-1$ colors. Suppose none of the candidate colors are valid for the edge $xy$. Then, for every color $\mu_i \in S$, there exist $(1,\mu_i,xy)$-critical paths. Also notice that, $2 \notin F_{yy_1}$. Suppose there exists a color $\mu_k \in S$ such that $\mu_k \notin F_{yy_2}$. Then we perform the following recoloring $c$ of $g$: the edge $yy_2$ is assigned the color $\mu_k$. By Lemma~\ref{lem:UniqueMaxBichPath},        at most one $(1,\mu_k)$-maximal bichromatic path can pass through $y$. We have already seen that $(1,\mu_k,xy)$-critical path exists. Therefore, a $(1,\mu_k,yy_2)$-critical path can not exist. So, the recoloring $c$ is valid. Assigning color $2$ to the edge $xy$ gives us a valid coloring for $G$ because $2 \notin F_{yy_1}$. Therefore,  $F_{yy_2} = S$. We now perform the recoloring $c'$ of $g$: the color of edges $yy_1$ and $yy_2$ are exchanged, i.e., $yy_1$ is assigned color $2$ and $yy_2$ is color $1$. The recoloring is proper because $1 \notin F_{yy_2}$ and $2 \notin F_{yy_1}$. The recoloring is valid because $2 \notin F_{yy_2}$ and $1 \notin F_{yy_1}$. Then, assigning any  color $\mu_i$  from the set $S$ to the edge $xy$ gives a valid coloring for $G$. This is because $(1,\mu_i,xy)$-critical path passes through $y_1$. By Lemma~\ref{lem:UniqueMaxBichPath}, at most one $(1,\mu_i)$-maximal bichromatic path can pass through $y$.

 Therefore, we can assume that $d_G(x) =3$ and $d_G(y) = 3$. Let $g(xx_1) = g(yy_1 ) =1$, $g(yy_2) = 2$ and $g(xx_2) = 3$. Let  $S$ contain $t -2$ candidate colors $\{\mu_{1}, \mu_2,...\mu_{t-2}\}$. Since no color in $S$ is valid for the edge $xy$, all colors in $S$ are also present in $F_{yy_1} $ and $F_{xx_1}$. Depending on whether $2 \in F_{yy_1}$ or not, we have the following subcases:

 \begin{subcase}
     $2 \in F_{yy_1}$
 \end{subcase}
 Then, $F_{yy_1} = \{2, \mu_1, \mu_2,...\mu_{t-2}\}$. Suppose $3 \notin F_{yy_2}$. We define the recoloring $h_1$ of $g$: The edge $yy_2$ is recolored with color $3$. The  coloring $h_{1}$ is valid as there is no $(1,3)$-bichromatic cycles formed due to recoloring as $3 \notin F_{yy_1}$. Notice that $|F_{xy} \cap F_{yx}| =2$ and the problem has reduced to Case~\ref{case:atmost3_intersection2}. Therefore $3 \in F_{yy_2}$.
 
 Suppose $1 \notin F_{yy_2}$. Since $3 \notin F_{yy_1}$ and $1 \notin F_{yy_2}$, we can define the following recoloring $h_2$ of $g$: the edge $yy_1$ is assigned the color $3$ and the edge $yy_2$ is assigned the color $1$. Notice that there is no $(1,3)$-bichromatic cycle created due to the recoloring, since $1 \notin F_{yy_1}$. Hence, the recoloring is valid and the problem has reduced to Case~\ref{case:atmost3_intersection2}. Therefore, we can say that $1  \in F_{yy_2}$ and $  3 \in F_{yy_2}$. Therefore, there exists a color $\mu_i \notin F_{yy_2}.$
 
 We can attempt the following recoloring $h_3$ of $g$: the edge $yy_2$ is recolored with color $\mu_i$. By Lemma~\ref{lem:UniqueMaxBichPath} at most one $(1, \mu_i)$-maximal bichromatic path can pass throught $y$ and we have seen that the $(1, \mu_i)$-maximal bichromatic path through $y$ goes to $x$. Therefore, recoloring is valid. We can attempt to extend the coloring $h_3$ to get a coloring for $G$, by assigning color $2$ to $xy$. This is invalid only if $(1,2, xy)$-critical path existed in coloring $g$ of $G'$. Therefore, $2 \in F_{xx_1}$ and $F_{xx_1} = \{2, \mu_1, \mu_2,...\mu_{t-2}\}$. We have the following two subcases:
\begin{subsubcase}
    $2 \notin F_{xx_2}$
\end{subsubcase}
We define the following recoloring $h_4$ of $g$: the edge $xx_2$ is assigned color $2$.  We have seen that $(1,2)$-maximal bichromatic path that passes through $x$ goes to $y$. By Lemma~\ref{lem:UniqueMaxBichPath}, at most one $(1,2)$-maximal bichromatic path can pass through $x$ . Therefore, recoloring is valid. Since $|F_{xy} \cap F_{yx}|  = 2$, the problem has reduced to Case~\ref{case:atmost3_intersection2}..

\begin{subsubcase}
    $2 \in F_{xx_2}$
\end{subsubcase}
Observe that all colors in $S$ must be present in $F_{xx_2}$. If that wasn't the case and a color $\mu_i \in S$ wasn't present in $ F_{xx_2}$, then we can perform the following recoloring $h_5$ of $g$: recolor the edge $xx_2$ with the color $\mu_i$. By Lemma~\ref{lem:UniqueMaxBichPath} at most one $(1, \mu_i)$-maximal bichromatic path can pass through $x$ and we have seen that the $(1, \mu_i)$-maximal bichromatic path through $x$ goes to $y$. Therefore, coloring is valid. Then we can extend this coloring to get a valid coloring for $G$ by assigning the newly freed color $3$ to the edge $xy$. This is a valid coloring for $G$ since $3 \notin F_{xx_1}$. Therefore, we arrive at a contradiction.

\begin{subcase}
    $2 \notin F_{yy_1}$
\end{subcase}

Suppose there exists a color $\mu_{i} \notin F_{yy_2}$. Then we define a recoloring $h_{6}$  from $g$ : color the edge $yy_2$ from $2$ to $\mu_{i}$. The resulting coloring $h_{6}$ is valid, because $(1, \mu_{i})$-critical path passing through $y$ goes to $x$.  By Lemma~\ref{lem:UniqueMaxBichPath}, there is no $(1, \mu_{i})$-bichromatic cycle passing through $y$ and $y_2$. We can extend the coloring $h_6$ to get a coloring of $G$ by assigning the color $2$ to edge $xy$. Since $2 \notin F_{yy_1}$, this gives us a valid coloring of $G$, which is a  contradiction to our assumption that $G$ is the minimum counterexample. Therefore, every $\mu_{i} \in F_{yy_2}$.
 
Suppose $1 \notin F_{yy_2}$. We can define a new recoloring $h_7$ of $g$: Exchange the colors of the edges $yy_1$ and $yy_2$, i.e., assign color $2$ to $yy_1$ and color $1$ to $yy_2$. Since $2 \notin F_{yy_2}$, there is no $(1,2)$-bichromatic cycle passing through $y_2$ due to the recoloring.  Hence the recoloring is valid.  Notice that $(1, \mu_{i})$-maximal bichromatic path that goes through $x$ goes to $y_1$. By Lemma~\ref{lem:UniqueMaxBichPath}, at most one such maximal bichromatic path can go through $x$. So, a $(1, \mu_{i},xy)$-critical path through $y_2$ can not exist. Therefore, assigning any color $\mu_i \in S$ to the edge $xy$ does not create any $(1, \mu_{i})$-bichromatic cycles that pass through $y_2$. This gives a valid coloring for $G$, a contradiction.  Therefore, we can conclude that $1 \in F_{yy_2}$. So, we can infer that $3 \notin F_{yy_2}$ and $F_{yy_2} = \{1, \mu_1, \mu_2, ..., \mu_{t -2} \}$.

Next, we attempt recoloring $h_8$ of coloring $g$: Recolor the edge $yy_2$ from color $2$ to $3$. Notice that $|F_{xy} \cap F_{yx}| = 2$ in $h_{8}$. If $h_8$ is valid, the problem reduces to Case~\ref{case:atmost3_intersection2}. Coloring $h_{8}$ is invalid only if there was a $(1,3, yy_2)$-critical path.

Suppose there exists a color $\mu_i \notin F_{xx_2}$. Then we can define recoloring $h_{9}$ of $g$: Recolor the edge $xx_2$ to color $\mu_i$. This is a valid coloring because by Lemma~\ref{lem:UniqueMaxBichPath}, at most one  $(1, \mu_i)$-maximal bichromatic path passes through $x$ and we know that a $(1, \mu_i, xy)$-critical path already exists. Since there exists a $(1,3, y, y_2)$-maximal bichromatic path, there can not exist a $(1,3, y,x)$-critical path. Therefore, it is valid to assign color $3$ to edge $xy$ to get a valid coloring for $G$, which is a  contradiction. Therefore, every color in $S$ is present in $F_{xx_2}.$

Now, we have the following subcases:
\begin{subsubcase}
    $2 \notin F_{xx_2}$
\end{subsubcase} 
 We define the following recoloring $h_{10}$ of $g$: The edge $xx_2$ is recolored as $2$. Notice that $|F_{xy} \cap F_{yx}| =2 $ and therefore it reduces to Case~\ref{case:atmost3_intersection2}, if recoloring is valid. The coloring $h_{10}$ is only invalid if the recoloring created a $(1,2)$-bichromatic cycle passing through $x$ and $x_2$. Therefore, we can conclude that $2 \in F_{xx_1}$ and $1 \in F_{xx_2}$. We can see that $F_{xx_1} = \{2,\mu_1, \mu_2, ..., \mu_{t-2}\} $ and $F_{xx_2} = \{1,\mu_1, \mu_2, ..., \mu_{t-2}\} $. 
 
 Now we attempt recoloring $h_{11}$ of coloring $g$: Recolor the edge $yy_1$ from color $1$ to color $2$ and recolor the edge $yy_2$ from color $2$ to color $3$.  Recoloring is proper because $3 \notin F_{yy_2}$ and $2 \notin F_{yy_1}$. Recoloring is valid because $2 \notin F_{yy_2}$. If any of the colors in $S$ are valid for $xy$, we are done. Therefore, we can assume that there exists $( 3, \mu_i,  xy)$-critical paths for every color $\mu_i$ in $S$. 
 Now let us attempt the following recoloring $h_{12}$ of $g$: The edge $xx_1$ is recolored with color $3$ and the edge $xx_2$ is recolored with color $2$. The recoloring is valid due to the absence of color $3$ in $F_{xx_2}$. We can attempt to extend this coloring to get a valid coloring for $G$ by assigning the edge $xy$ with colors from $S$. This assignment is invalid if for every $\mu_i \in S$, there exists a $( 2, \mu_i, xy)$-critical path in coloring $g$ of $G'$. 
 From this we can conclude that $F_{x_2w_i} = \{2,3\}$, for every neighbor $w_i \in N_G(x_2) \setminus \{x,x^*\}$, where $x^* \in N_G(x_2)$ and $g(x_2x^*) \notin S$.
Next, we attempt the following recoloring $h_{13}$ of $g$: The edge $xx_1$ is recolored with color $3$ and the edge $xx_2$ is recolored with color $2$ (just as before). In addition, we also exchange the colors of the edges $x_2w_i$ and $x_2w_j$, where $i \neq j$ and $i, j \leq t - 2$ and $g(x_2w_i) = \mu_i$ and $g(x_2w_j) = \mu_j$, with $\mu_i, \mu_j \in S$ (Since $\Delta(G) \geq 4$, $\mu_i, \mu_j$ exists). This breaks the $(2,\mu_i,  x, y)$ and $(2, \mu_j,  xy)$-critical path that was present. Therefore, either of $\mu_i$ or $\mu_j$ is a valid color for $xy$ and we can get a valid coloring for $G$. Note that exchanging the colors of the edges $x_2w_i$ and $x_2w_j$ results in a valid coloring of $G'$ because $\{2,3\} \notin F_{xx_2}$ and because $F_{x_2w_i} = \{2,3\}$, for every $w_i$.   

\begin{subsubcase}
    $ 2 \in F_{xx_2}$
\end{subsubcase}
Since $2 \in F_{xx_2}$, $1 \notin F_{xx_2}$. Suppose $2 \in F_{xx_1}$. We define the following recoloring $h_{14}$ of $g$: Recolor the edge $xx_1$ with color $3$ and recolor the edge $xx_2$ with color $1$. The coloring is valid since recoloring doesn't create a $(1, 3)$-bichromatic cycle due to absence of color $1$ in $F_{xx_1}$. Coloring the edge $xy$ with any color from $S$ gives a valid coloring for the graph $G$. Notice that  $(1,\mu_i)$-maximal bichroamtic path that pass through $y$ go to $x_1$. By Lemma~\ref{lem:UniqueMaxBichPath}, at most one $(1,\mu_i)$-maximal bichromatic path can pass through $y$ and therefore, we can see that recoloring is valid.

Therefore, we can conclude that $2 \notin F_{xx_1}$. We perform the following recoloring $h_{15}$ of $g$: The edge $xx_1 $ is recolored with color $2$ and the edge $xx_2$ is recolored with color $1$. This recoloring is valid since no $(1,2)$-bichromatic cycle is created due to the absence of color $1$ in $F_{xx_1}$. Notice that $|F_{xy} \cap F_{yx}| = 2$ and this reduces to Case~\ref{case:atmost3_intersection2}.
\end{proof}
\renewcommand\qed{$\hfill\blacksquare$}

This concludes the proof for Lemma~\ref{lem:ge3}. Therefore, we may assume either $d_G(x) >3 $ or $d_G(y) >3$. Without loss of generality, assume $d_G(y) >3$. 

\setcounter{case}{0}
\renewcommand\thecase{\alph{case}}
		\begin{case}\label{case:3SparseNoIntersection}
			$|F_{xy} \cap F_{yx}| = 0$.
		\end{case}
		
		In this case, any candidate color for the edge $xy$ is also valid. Since $|F_{xy} \cup F_{yx}| \leq t$, there is at least one candidate color in $S$ that can be assigned to the edge $xy$ to get a valid coloring $f$ of $G$, contradiction to the fact that $G$ is a minimum counterexample. 
		
		\begin{case}\label{case:3SparseOneIntersection}
			$|F_{xy} \cap F_{yx}| = 1$.
		\end {case}
            Without loss of generality, assume that $\alpha$ is the color present in both $F_{xy}$ and $F_{yx}$. Let $y_1 \in N(y)$ and $g(yy_1) = \alpha$. Since $|F_{xy} \cap F_{yx}| = 1$ and since $d_G(x) + d_G(y) < t +3$, we have $|F_{xy} \cup F_{yx}| \leq t -1$, implying that $|S| \ge 2$. Let $\mu,\nu \in S$. Since any color in $S$ is not valid for the edge $xy$ in $G$, there exists an $(\alpha,\mu,xy)$-critical path and an $(\alpha,\nu,xy)$-critical path in $G$. Let $y'_1$ and $y''_1$ be the other neighbors of $y_1$.  Without loss of generality, let $g(y_1y'_1) = \mu$ and $g(y_1y''_1) = \nu$.
			
		\begin{subcase}\label{subcase:3sparseOneIntersection_1}
			There exists $y'$ in $N(y) \setminus y_1$ such that $g(yy') = \gamma$ for some $\gamma \in F_{xy}$, and the remaining colors on $y'$ are $\mu$, $\nu$.
		\end{subcase}
			
		Now, perform a color exchange on the edges $yy_1$ and $yy'$, i.e., assign $\gamma$ to $yy_1$ and $\alpha$ to $yy'$. Clearly, the resulting coloring is valid for $G'$ because $\mu,\nu \notin F_{xy}$. Observe that the $(\alpha,\mu)$-maximal bichromatic path and the $(\alpha,\nu)$-maximal bichromatic path starting from $x$ ends at the vertex $y_1$ because of the recoloring. Therefore, by Lemma~\ref{lem:UniqueMaxBichPath}, either $\mu$ or $\nu$ can be assigned to the edge $xy$ to get a valid coloring $f$ of $G$, a contradiction to the fact that $G$ is a minimum counterexample.
		
\begin{figure}[!h] 
            \begin{subfigure}[b]{0.45\textwidth}
		
			\begin{tikzpicture}[scale = 0.6]
				\begin{scope}[every node/.style={circle,draw,inner sep=0pt, minimum size=3ex }]
					\node (x) at (1,0) {$x$};
					\node (y) at (4,0) {$y$};
					\node (x1) at (-1,1) {$x_1$};
					\node (x2) at (-1,-1) {$x_2$};
					
					\node (y1) at (5,2) {$y_1$};
 				   \node (y2) at (6.25,-0.5) {$y_{\Delta-1}$};
                   \node (y3) at (4,3.5){$y'_1$};
                   \node (y4) at (5,4){$y''_1$};

                 \node (y9) at (7,2.3) {$y'$};
				\end{scope}
				  \node (x3) at (0.9,1.9) {};
				   \node (x4) at (0.1,3.4) {};
				  \node  (y5) at (3.5,4.7){};
                  \node  (y6) at (4.5,6){};
                    \node  (y7) at (4.6,1.5){};
                    
				\node  (y8) at (5,-0.4){};

                \node  (y21) at (7.1,0.9){};
                \node  (y20) at (6.8,0.4){};
                \draw [dash dot](y20) to (y);
                \draw [dash dot](y21) to (y);

				\draw [dotted](x) to (y);
				\draw (x) to node[above]{$\alpha$}(x1);
				\draw (x) to (x2);
				\draw (y) to  node[above]{$\alpha$}(y1);
				\draw (y) to (y2);
				\draw (y1) to node[above]{$\mu$} (y3);
				\draw (y1) to node[right]{$\nu$} (y4);
				\draw (y3) to  node[right]{$\alpha$}(y5);
				\draw (y4) to  node[right]{$\alpha$}(y6);
				
				\draw (x1) to node[above]{$\mu$} (x3);
				\draw (x1) to node[above]{$\nu$} (x4);
				\draw  [dashed, , bend right] (y5) to node[above]{} (x3);
				\draw  [dashed, , bend right] (y6) to node[above]{} (x4);
                   \node (y11) at (7,3.5){ };
                   \node (y10) at (7.5,4){};
                   \draw(y) to node[above]{$\gamma$}(y9);
                   \draw (y9) to node[left]{$\mu$} (y11);
				\draw (y9) to node[right]{$\kappa$} (y10);
			\end{tikzpicture}%
			\caption{Coloring $g$ of $G'$. Critical paths $(\alpha,\mu,x,y)$ and $(\alpha,\nu,x,y)$ exist.}
			\label{fig:color_g}
		\end{subfigure}\hspace{1 cm}\begin{subfigure}[b]{0.45\textwidth}
        \centering
        \begin{tikzpicture}[scale = 0.6]
				\begin{scope}[every node/.style={circle,draw,inner sep=0pt, minimum size=3ex}]
					\node (x) at (1,0) {$x$};
					\node (y) at (4,0) {$y$};
					\node (x1) at (-1,1) {$x_1$};
					\node (x2) at (-1,-1) {$x_2$};
					
					\node (y1) at (5,2) {$y_1$};
 				  \node (y2) at (6.25,-0.5) {$y_{\Delta -1}$};
                   \node (y3) at (4,3.5){$y'_1$};
                   \node (y4) at (5,4){$y''_1$};
                   
                   \node(y12) at (7.7, 1.5) {$y''$};
                 \node (y9) at (7,2.3) {$y'$};
				\end{scope}
				  \node (x3) at (0.9,1.9) {};
				   \node (x4) at (0.1,3.4) {};
				  \node  (y5) at (3.5,4.7){};
                  \node  (y6) at (4.5,6){};
                    \node  (y7) at (4.6,1.5){};
                    \node (y13) at (8.7, 1.8){};

                  \node  (y8) at (5,-0.4){};
				\draw [dotted](x) to (y);
				\draw (x) to node[above]{$\alpha$}(x1);
				\draw (x) to (x2);
				\draw (y) to  node[above]{$\gamma$}(y1);
				\draw (y) to (y2);
				\draw (y1) to node[above]{$\mu$} (y3);
				\draw (y1) to node[right]{$\nu$} (y4);
				\draw (y3) to  node[right]{$\alpha$}(y5);
				\draw (y4) to  node[right]{$\alpha$}(y6);
				
				\draw (x1) to node[above]{$\mu$} (x3);
				\draw (x1) to node[above]{$\nu$} (x4);
				\draw  [dashed, , bend right] (y5) to node[above]{} (x3);
				\draw  [dashed, , bend right] (y6) to node[above]{} (x4);
                   \node (y11) at (7,3.5){ };
                   \node (y10) at (7.5,4){};
                   \draw (y) to node[above]{$\alpha$}(y9);
                   \draw (y9) to node[left]{$\mu$} (y11);
				\draw (y9) to node[right]{$\kappa$} (y10);
				\draw (y) to node[above]{$\kappa$} (y12);
				 \draw (y12) to node[above]{$\alpha$}(y13);
				 \draw  [dashed, bend right] (y13) to node[above]{} (y10);
				\node  (y20) at (6.8,0.4){};
                \draw [dash dot](y20) to (y);
				 
			\end{tikzpicture}\caption{Coloring $g_1$ is invalid only if it forms a $(\alpha, \kappa)$-bichromatic cycle as shown.}
			\label{fig:color_g1}
    \end{subfigure}\hspace{1cm}
    \begin{subfigure}[b]{0.45\textwidth}
        \centering
        \begin{tikzpicture}[scale = 0.6]
				\begin{scope}[every node/.style={circle,draw,inner sep=0pt, minimum
				 size=3ex}]
				
					\node (x) at (1,0) {$x$};
					\node (y) at (4,0) {$y$};
					\node (x1) at (-1,1) {$x_1$};
					\node (x2) at (-1,-1) {$x_2$};
					
					\node (y1) at (5,2) {$y_1$};
 				  \node (y2) at (6.25,-0.5) {$y_{\Delta-1}$};
                   \node (y3) at (4,3.5){$y'_1$};
                   \node (y4) at (5,4){$y''_1$};
                   
                   \node(y12) at (7.7, 1.5) {$y''$};
                 \node (y9) at (7,2.3) {$y'$};
				\end{scope}
				  \node (x3) at (0.9,1.9) {};
				   \node (x4) at (0.1,3.4) {};
				  \node  (y5) at (3.5,4.7){};
                  \node  (y6) at (4.5,6){};
                    \node  (y7) at (4.6,1.5){};
                    \node (y13) at (9.2, 1.8){};
                   \node (y14) at (8.3, 2.5){};
                  \node  (y8) at (5,-0.4){};
				\draw [dotted] (x) to (y);
				\draw (x) to node[above]{$\alpha$}(x1);
				\draw (x) to (x2);
				\draw (y) to  node[above]{$\gamma$}(y1);
				\draw (y) to (y2);
				\draw (y1) to node[above]{$\mu$} (y3);
				\draw (y1) to node[right]{$\nu$} (y4);
				\draw (y3) to  node[right]{$\alpha$}(y5);
				\draw (y4) to  node[right]{$\alpha$}(y6);
				
				\draw (x1) to node[above]{$\mu$} (x3);
				\draw (x1) to node[above]{$\nu$} (x4);
				\draw  [dashed, , bend right] (y5) to node[above]{} (x3);
				\draw  [dashed, , bend right] (y6) to node[above]{} (x4);
			 
                   \node (y11) at (7,3.5){ };
                   \node (y10) at (7.5,4){};
                   \draw (y) to node[above]{$\nu$}(y9);
                   \draw (y9) to node[left]{$\mu$} (y11);
				\draw (y9) to node[right]{$\kappa$} (y10);
				\draw (y) to node[above]{$\kappa$} (y12);
				 \draw(y12) to node[above]{$\alpha$}(y13);
				  \draw(y12) to node[above]{$\nu$}(y14);
				 \draw  [dashed, bend right] (y13) to node[above]{} (y10);
				  \draw  [dashed, thick, bend left] (y10) to node[above]{} (y14);
	
				\node  (y20) at (6.8,0.4){};
                \draw [dash dot](y20) to (y);
			\end{tikzpicture}\caption{Coloring $g2$ is invalid only if  it forms a $(\kappa, \nu)$-bichromatic cycle as shown.}
			\label{fig:color_g2} 
    \end{subfigure}	\hspace{0.75cm}
    \begin{subfigure}[b]{0.45\textwidth}
        \centering
        \begin{tikzpicture}[scale = 0.6]
				\begin{scope}[every node/.style={circle,draw,inner sep=0pt, minimum
				 size=3ex}]
	
\node (x) at (1,0) {$x$};
					\node (y) at (4,0) {$y$};
					\node (x1) at (-1,1) {$x_1$};
					\node (x2) at (-1,-1) {$x_2$};
					
					\node (y1) at (5,2) {$y_1$};
 				  \node (y2) at (6.25,-0.5) {$y_{\Delta-1}$};
                   \node (y3) at (4,3.5){$y'_1$};
                   \node (y4) at (5,4){$y''_1$};
                   
                   \node(y12) at (7.7, 1.5) {$y''$};
                 \node (y9) at (7,2.3) {$y'$};
				\end{scope}
				  \node (x3) at (0.9,1.9) {};
				   \node (x4) at (0.1,3.4) {};
				  \node  (y5) at (3.5,4.7){};
                  \node  (y6) at (4.5,6){};
                    \node  (y7) at (4.6,1.5){};
                    \node (y13) at (9.2, 1.8){};
                   \node (y14) at (8.3, 2.5){};
                  \node  (y8) at (5,-0.4){};
				\draw [dotted] (x) to (y);
				\draw (x) to node[above]{$\alpha$}(x1);
				\draw (x) to (x2);
				\draw (y) to  node[above]{$\kappa$}(y1);
				\draw (y) to (y2);
				\draw (y1) to node[above]{$\mu$} (y3);
				\draw (y1) to node[right]{$\nu$} (y4);
				\draw (y3) to  node[right]{$\alpha$}(y5);
				\draw (y4) to  node[right]{$\alpha$}(y6);
				
				\draw (x1) to node[above]{$\mu$} (x3);
				\draw (x1) to node[above]{$\nu$} (x4);
	 
                   \node (y11) at (7,3.5){ };
                   \node (y10) at (7.5,4){};
                   \draw(y) to node[above]{$\alpha$}(y9);
                   \draw (y9) to node[left]{$\mu$} (y11);
				\draw (y9) to node[right]{$\kappa$} (y10);
				\draw (y) to node[above]{$\gamma$ } (y12);
				 \draw (y12) to node[above]{$\alpha$}(y13);
				  \draw (y12) to node[above]{$\nu$}(y14);
				\node  (y20) at (6.8,0.4){};
                \draw [dash dot](y20) to (y);
			\end{tikzpicture}\caption{Coloring $g_3$ is valid.  Edge $xy$ can be assigned color $\nu$.}
			\label{fig:color_g3}
    \end{subfigure}	
    
    \caption{Case b.2: There exists $y'$ in $N(y) \setminus y_1$ such that $g(yy') = \gamma$ for some $\gamma \in F_{xy}$, and exactly one of $\mu$, $\nu$ is present on $y'$.}
    
		\end{figure}
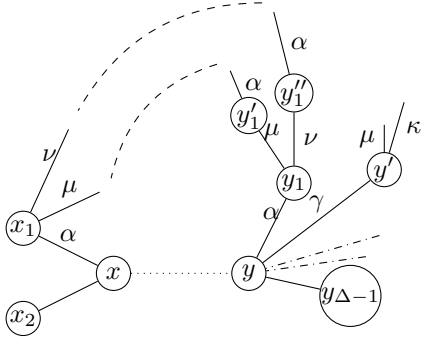
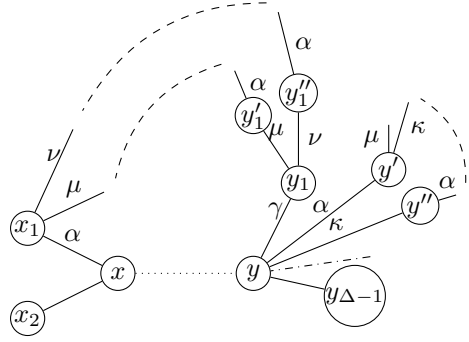
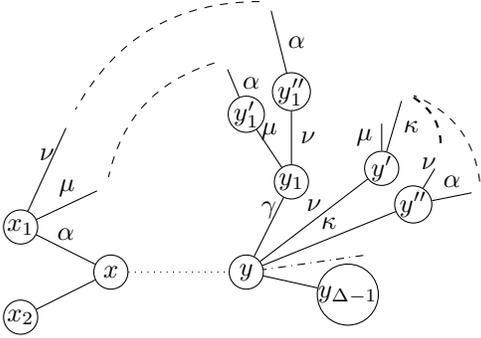
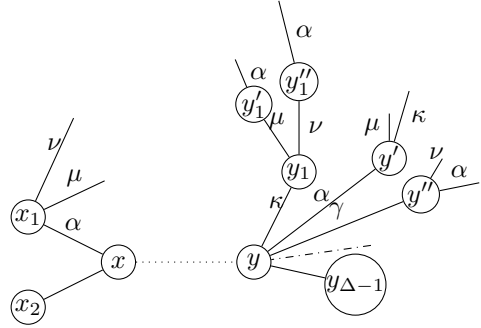	
				
		\begin{subcase}\label{subcase:3sparseOneIntersection_2}
			There exists $y'$ in $N(y) \setminus y_1$ such that $g(yy') = \gamma$ for some $\gamma \in F_{xy}$, and exactly one of $\mu$, $\nu$ is present on $y'$.
		\end{subcase}
		
		Without loss of generality, let the color $\mu$ be present (and not $\nu$) on some edge incident on $y'$. Let $\kappa$ be the color on the other edge incident on $y'$ (if such an edge is present). We first show that $\kappa$ can not be $\alpha$. Suppose $\kappa$ is $\alpha$. Then perform a recoloring $g_0$ of $g$: the edge $yy'$ is assigned the color $\nu$. Recall that $(\alpha, \nu, xy)$-critical path exists. Therefore, by Lemma~\ref{lem:UniqueMaxBichPath},  $(\alpha, \nu, yy')$-critical path can not exist. Therefore, recoloring $g_0$ is valid. Since $\gamma \notin F_{yy_1}$, the color $\gamma$ can be assigned to the edge $xy$ to get a valid coloring of $G$. Therefore, we can conclude that $\kappa$ can not be $\alpha$. Now, perform a color exchange on the edges $yy_1$ and $yy'$, i.e., assign $\gamma$ to $yy_1$ and $\alpha$ to $yy'$ to get a coloring $g_1$ of $G'$. If $g_1$ is a valid coloring, we can assign the color $\nu$ to the edge $xy$ to get a valid coloring of $G$, a contradiction. The coloring $g_1$ is not valid only if an $(\alpha, \kappa)$-bichromatic cycle is created when the edge $yy'$ was recolored with $\alpha$ (See Figure~\ref{fig:color_g1}). Such a cycle must pass through some edge $yy''$ colored $\kappa$. Thus we can also assume that one of the other edges incident on $y''$ is colored $\alpha$ in $g$.
			
		Now, we attempt a recoloring $g_2$ of $g$: Let the edges $yy_1$ and $yy'$ be recolored with $\gamma$ and $\nu$ respectively. If $g_2$ is valid, assignment of color $\mu$ for the edge $xy$ will result in a valid coloring for $G$, a contradiction. The coloring $g_2$ is not valid only if it creates a $(\kappa, \nu)$-bichromatic cycle passing through the vertices $y'$ and $y''$ (See Figure~\ref{fig:color_g2}). Therefore, we can assume that the last edge incident on $y''$ is colored with color $\nu$ in coloring $g$. Effectively, the colors seen at the vertex $y''$ in $g$ are $\kappa$, $\alpha$ and $\nu$.
			
		Now we attempt another recoloring $g_3$ of $g$: The edges $yy_1$, $yy'$ and $yy''$ are recolored with the colors $\kappa$, $\alpha$ and $\gamma$ respectively (See Figure~\ref{fig:color_g3}). Since $\mu,\nu \notin F_{xy}$, recoloring $yy_1$ with $\kappa$ doesn't produce any bichromatic cycles. Further since $\alpha \notin F_{yy_1}$, recoloring $yy'$ with $\alpha$ doesn't produce an $(\alpha,\kappa)$-bichromatic cycle. Since $\nu \notin F_{y''y}$ and $\gamma \notin F_{yy'}$, recoloring $yy''$ with $\gamma$ doesn't produce any bichromatic cycles. Hence, $g_3$ is a valid coloring for $G'$. Since $\nu \notin F_{yy'}$, we can get a valid coloring $f$ of $G$ from $g_3$ of $G'$ by assigning the color $\nu$ to the edge $xy$, a contradiction.

		\begin{subcase}\label{subcase:3sparseOneIntersection_3}
			No vertex in $N(y) \setminus y_1$ has an edge incident on it colored $\mu$ or $\nu$.
		\end{subcase}
		
		Consider any vertex $y' \in N_{G'}(y) \setminus y_1$. Let $g(yy') = \gamma$ for some $\gamma \in F_{xy}$. Let the other two edges incident on $y'$ (if present) have colors $\eta$ and $\zeta$. Now, consider the following color exchange $g'$: $g'(yy_1) = \gamma$ and $g'(yy') = \mu$. The only reason $g'$ fails to be valid is if it created a $(\mu, \eta)$ or a $(\mu, \zeta)$-bichromatic cycle that passes through some edge $yy''$ colored either $\eta$ or $\zeta$. But $y''$ can not have the colors $\mu$ or $\nu$ on the other edges incident on it by our assumption. Therefore the recoloring $g'$ must be valid. Assign $\nu$ to the edge $xy$ to get a valid coloring $f$ of $G$ from $g'$ of $G'$, a contradiction.
		
		
		\begin{case}\label{case:3SparseTwoIntersections}
			$|F_{xy} \cap F_{yx}| = 2$.
		\end{case}
		
		Let $y_1, y_2 \in N(y)$ with $g(yy_1) = \alpha$ and $g(yy_2) = \beta$. Let $y'_1,y''_1$ be the neighbors of $y_1$ and $y'_2,y''_2$ be the neighbors of $y_2$ other than $y$. Clearly, all the colors in $S$ must have been assigned to some edge in the set $R$ = $\{y_1y'_1, y_1y''_1, y_2y'_2, y_2y''_2\}$. If not, such a color is free for assignment to the edge $xy$ to get a valid coloring $f$ from $g$, a contradiction. Since $|F_{xy} \cap F_{yx}| = 2$, we have $|F_{xy} \cup F_{yx}| \leq t-2$ implying that $|S| \ge 3$. Let $\gamma,\eta,\mu \in S$. Hence, at least three edges in $R$ are assigned colors from $S$.
		
		Let $\nu$ be the color used for the fourth edge in $R$ and it can belong to either $S$ or $F_{xy}$. Without loss of generality, let $\gamma$ and $\eta$ be the colors assigned to edges $y_1y'_1$ and $y_1y''_1$ in coloring $g$ and let $\mu$ and $\nu$ be the colors assigned to $y_2y'_2$ and $y_2y''_2$. Observe that $\nu$ can be same as $\gamma$ or $\eta$ but not $\mu$, by the proper coloring restriction. Now, we recolor the edge $yy_1$ with color $\mu$. Since $\gamma,\eta \in S$, we have $\gamma,\eta \notin F_{xy}$. Hence, no bichromatic cycle is created by the recoloring implying that the resulting coloring is valid. If the $d(y_2)=2$ (and therefore $\nu$ is absent in $F_{yy_2}$) or if $\nu \in F_{xy}$, then any color in $\{\gamma, \eta\}$ can be assigned to the edge $xy$ to extend the coloring to get a valid coloring $f$ of $G$, a contradiction. Otherwise, we have $\nu \in S$. Without loss of generality, let $\nu = \gamma$. Then we can assign $\eta$ to the edge $xy$ to get a valid coloring $f$ of $G$, again a contradiction to the fact that $G$ is a minimum counterexample.
		
		In any case, we arrive at a contradiction implying that a minimum counterexample to Theorem~\ref{thm:ACI3sprs} does not exist. Therefore, Theorem~\ref{thm:ACI3sprs} holds and Corollary~\ref{cor:ACI3sprs} follows from the same.
	\end{proof}
	\section{Conclusion}
        We conclude our discussion on the acyclic chromatic index of $3$-sparse graphs by reiterating Theorem~\ref{thm:ACI3sprs} and Corollary~\ref{cor:ACI3sprs}. Any $3$-sparse graph $G$ can be acyclically edge colored with $\Delta+2$ colors. Further, if $G$ has an edge $uv$ with $d_G(u) + d_G(v) < \Delta + 3$, then there exists an acyclic edge coloring of $G$ with at most $\Delta +1$ colors. The acyclic edge coloring conjecture gives an upper bound of $\Delta+2$ for any arbitrary graph $G$. Hence, one can try to prove the conjecture which constitutes a nice long term research problem. More feasible directions would be to study acyclic edge coloring for bipartite graphs or highly structured class of complete graphs for which the problem is open. Looking for an efficient algorithm for the acyclic edge coloring of these graph classes may add value to the existing structural results in the literature.
	
	\bibliographystyle{SK}
	\bibliography{Acyclic3sparse}
	
\end{document}